\documentclass[11pt]{amsart}

\usepackage[USenglish]{babel}
\usepackage[T1]{fontenc} 
\usepackage[utf8,latin1]{inputenc}
\usepackage{tikz}
\usepackage{amsmath,amsthm,amssymb,amsfonts,mathrsfs}
\usepackage{dsfont}
\usepackage{mathtools}
\usepackage{graphicx}

\DeclareMathAlphabet{\mathbbold}{U}{bbold}{m}{n} 

\usepackage{enumitem}
\usepackage{nicematrix}
\usepackage{lmodern}
\usepackage[all]{xy}

\usepackage[bookmarks=true]{hyperref}
\usepackage{xcolor}
\hypersetup{
    colorlinks,
    linkcolor={red!50!black},
    citecolor={blue!50!black},
    urlcolor={blue!80!black},
}

\usepackage{todonotes}
\usepackage[bottom=3cm, top=3cm, left=3.5cm, right=3.5cm]{geometry}

\mathtoolsset{showonlyrefs}  

\newcommand{\R}{\mathbb{R}}

\newcommand{\N}{\mathbb{N}}

\newcommand{\eps}{\varepsilon}
\newcommand{\di}{\mathsf d}
\newcommand{\X}{{\rm X}}
\newcommand{\de}{\mathrm{d}}

\newcommand{\rest}[2]{{{\rm restr}_{#1}^{#2}}}
\DeclareMathOperator{\Geo}{Geo}
\newcommand{\p}{\mathtt p}
\newcommand{\Prob}{\mathscr{P}}
\newcommand{\optgeo}{\mathrm{OptGeo}}
\newcommand{\opt}{\mathrm{Opt}}
\newcommand{\spt}{\mathrm{spt}}
\newcommand{\diam}{\mathrm{diam}}

\newcommand{\cd}{{\sf CD}}
\newcommand{\mcp}{{\sf MCP}}
\newcommand{\sbm}{{\sf SBM}}
\newcommand{\bm}{{\sf BM}}
\newcommand{\m}{\mathfrak{m}}

\newcommand{\E}{\mathcal{E}}

\definecolor{myblue}{RGB}{18,10,170}
\definecolor{myred}{RGB}{230,31,27}
\definecolor{mygreen}{RGB}{31,139,34}

\theoremstyle{plain}
\newtheorem{thm}{Theorem}[section]
\newtheorem{cor}[thm]{Corollary}
\newtheorem{lem}[thm]{Lemma}
\newtheorem*{lem*}{Lemma}
\newtheorem{prop}[thm]{Proposition}

\theoremstyle{definition}
\newtheorem{defn}[thm]{Definition}

\theoremstyle{remark}
\newtheorem{rmk}[thm]{Remark}
\newtheorem*{notation}{Notation}

\newcommand{\sM}{\mathscr M}
\author[M. Magnabosco]{Mattia Magnabosco}
\author[L. Portinale]{Lorenzo Portinale}
\author[T. Rossi]{Tommaso Rossi}
\address{Institut f\"ur Angewandte Mathematik, Universit\"at Bonn, Bonn, Germany}
\email{\href{mailto:magnabosco@iam.uni-bonn.de}{magnabosco@iam.uni-bonn.de}}
\email{\href{mailto:rossi@iam.uni-bonn.de}{rossi@iam.uni-bonn.de}}
\email{\href{mailto:portinale@iam.uni-bonn.de}{portinale@iam.uni-bonn.de}}

\title[The strong {B}runn--{M}inkowski inequality implies the {$\cd$} condition]{The strong {B}runn--{M}inkowski inequality and its equivalence with the  {$\cd$} condition}

\date{\today}

\begin{document}

\begin{abstract}
In the setting of essentially non-branching metric measure spaces, we prove the equivalence between the curvature dimension condition $\cd(K,N)$, in the sense of Lott--Sturm--Villani \cite{MR2237206,MR2237207,MR2480619}, and a newly introduced notion that we call strong Brunn--Minkowski inequality $\sbm(K,N)$.
This condition is a reinforcement of the generalized Brunn--Minkowski inequality $\bm(K,N)$, which is known to hold in $\cd(K,N)$ spaces. Our result is a first step towards providing a full equivalence between the $\cd(K,N)$ condition and the validity of $\bm(K,N)$, which has been recently proved in \cite{seminalpaper} in the framework of weighted Riemannian manifolds.
\end{abstract}

\maketitle


\section{Introduction}
In their seminal papers \cite{MR2237206,MR2237207,MR2480619}, Lott--Sturm--Villani introduced a synthetic notion of
curvature dimension bounds, in the non-smooth setting of metric measure spaces, usually denoted by $\cd(K,N)$, with $K\in\R$, $N\in (1,\infty]$. This property is formulated using the theory of \emph{optimal transport} and, in the setting of Riemannian manifolds, is equivalent to having Ricci curvature bounded below and dimension bounded above. More precisely, the $\cd(K,N)$ condition consists in a convexity property of the so-called R\'enyi entropy functional along Wasserstein geodesics. 

In the Riemannian setting, curvature bounds are sufficient to deduce many geometric and functional inequalities. Similar results can be obtained in the framework of a metric measure space $(\X,\di,\m)$ as a consequence of the curvature-dimension condition. A celebrated geometric inequality that can be deduced from the $\cd(K,N)$ condition, in a generalized form, is the \emph{Brunn--Minkowski inequality}: the convexity of the R\'enyi entropy translates to a concavity property of the mass of the $t$-midpoints, namely of the function
\begin{equation}
   [0,1]\ni t\mapsto \m(M_t(A,B))^{\frac1N},\qquad A,B\subset\X\ \text{Borel sets}, 
  \end{equation}
see Definition \ref{def:BRUNOOO}. This was already observed in the first papers on the $\cd(K,N)$ condition, see in particular \cite[Proposition 2.1]{MR2237207}.
On the one hand, a remarkable feature of the Brunn--Minkowski inequality is that its formulation does not invoke optimal transport. On the other hand, its proof relies on the well-known inclusion of the support $D_t(A,B)$ of the (unique) Wasserstein $t$-midpoint, between the normalized uniform distributions on $A$ and $B$, in the set of $t$-midpoints $M_t(A,B)$. Therefore, a natural strengthening of the Brunn--Minkowski inequality is to require that the aforementioned concavity property holds, not for the whole set of $t$-midpoints, but only for the support of the Wasserstein $t$-midpoint. This leads to main novelty of this paper: the introduction of the \emph{strong Brunn--Minkowski inequality}, which we denote by $\sbm(K,N)$, cf. Definition~\ref{def:forte_bruno}. Despite being still dependent on the optimal transport, this inequality is reminiscent of the Brunn--Minkowski one. Our main result is that $\sbm(K,N)$ is equivalent to $\cd(K,N)$, in the setting of essentially non-branching metric measure spaces. We refer to Sections \ref{sec:prelim} and \ref{sec:cd_cond_bruno} for the precise definitions. 

\begin{thm}
\label{thm:SBM_imnplies_CD}
Let $K\in\R$ and $N>1$ and let $(\X,\di,\m)$ be an essentially non-branching metric measure space supporting $\sbm(K,N)$. Then, $(\X,\di,\m)$ is a $\cd(K,N)$ space. In particular, 
\begin{center}
    $(\X,\di,\m)$ supports $\sbm(K,N)$ if and only if it satisfies $\cd(K,N)$.
\end{center}
\end{thm}

This theorem is a first step towards the complete equivalence between
the Brunn--Minkowski inequality and the curvature dimension condition, in the setting of essentially non-branching spaces. The interest in the aforementioned equivalence is twofold: on one side, it would provide a characterization of the curvature dimension condition \emph{without} the need of optimal transport. On the other side, it would provide an alternative proof of the \emph{globalization theorem}, cf. \cite{MR4309491}. Indeed, according to \cite[Theorem~1.2]{MR3608721}, the local curvature dimension condition, denoted by $\cd_{\mathrm{loc}}(K,N)$, is enough to the deduce the (global) Brunn--Minkowski inequality with sharp coefficients. Note that, should $\cd_{\mathrm{loc}}(K,N)$ imply $\sbm(K,N)$, Theorem~\ref{thm:SBM_imnplies_CD} would already be enough to deduce the globalization theorem. However, this implication does not easily follow from the arguments in \cite[Theorem~1.2]{MR3608721}. In fact, their technique is based on a suitable localization argument, built upon the $L^1$-optimal transport and thus it does not immediately imply the validity of the strong Brunn--Minkowski inequality, which on the contrary is based on the $L^2$-optimal transport.

In \cite{seminalpaper}, we prove the equivalence between the Brunn--Minkowski inequality and the curvature dimension condition, in the setting of weighted Riemannian manifolds. In this framework, using the \emph{full} Riemann curvature tensor, we are able to identify pairs of sets $A$ and $B$, for which the sets $M_t(A,B)$ and $D_t(A,B)$ are comparable in measure. In the setting of essentially non-branching metric measure spaces, the relation between $D_t(A,B)$ and $M_t(A,B)$ is in general less clear. A better understanding of this problem is most likely required to close the gap between the Brunn--Minkowski inequality and the $\cd(K,N)$ condition, cf. Section \ref{sec:Dt_VS_Mt_emph} for a more detailed discussion.

\subsection*{Organization of the paper} In Section \ref{sec:prelim}, we introduce the basic tools of metric measure spaces and optimal transport. In Section \ref{sec:cd_cond_bruno}, we recall the definition of a $\cd(K,N)$ space and we prove an equivalence criterion, which we believe is of independent interest, cf. Proposition \ref{prop:toprovecd}. Finally, we introduce the strong Brunn--Minkowski inequality $\sbm(K,N)$, cf. Definition \ref{def:forte_bruno}, and we show its interplay with the Brunn--Minkowski inequality. In Section \ref{sec:main_theorem}, we prove Theorem \ref{thm:SBM_imnplies_CD}: the most challenging step is to find a large class of measures for which $\sbm(K,N)$ implies $\cd(K,N)$, cf. Theorem \ref{thm:cdwithstep}. At last, in Section \ref{sec:finndelany_grandeacquisto}, we show that the Brunn--Minkowski inequality implies the so-called measure contraction property, cf. Proposition \ref{prop:capitantadda}, and we add some final remarks.

\subsection*{Acknowledgments} 
L.P. acknowledges support from the Hausdorff Center for Mathematics in
Bonn. T.R. acknowledges support from the Deutsche Forschungsgemeinschaft (DFG, German Research Foundation) through the collaborative research centre ``The mathematics of emerging effects'' (CRC 
1060, Project-ID 211504053). The authors are grateful to Prof. Karl-Theodor Sturm for fruitful discussions on the topic.

\section{Preliminaries}
\label{sec:prelim}
A metric measure space is a triple $(\X,\mathsf{d},\mathfrak{m})$, where $(\X,\mathsf{d})$ is a Polish metric space (i.e.\ complete and separable) and $\mathfrak{m}$ is a positive Borel measure on $\X$, finite on bounded sets. We call $\mathscr M_+ (\X)$ the set of all positive and finite Borel measures on a metric space $(\X,\di)$.
We denote by $\mathscr{P}(\X)\subset \mathscr M_+ (\X)$ the set of all Borel probability measures on $(\X,\di)$, and by $\mathscr{P}_2(\X) \subset \mathscr{P}(\X)$ the set of all probability measures with finite second moment. Moreover, the set of probability measures in $\Prob_2(\X)$, which are absolutely continuous with respect to the reference measure $\m$, will be denoted by $\Prob^{ac}(\X,\m)$. On the space $\mathscr{P}_2(\X)$, we introduce the 2-Wasserstein distance
\begin{equation}\label{eq:W2}
    W_2^2(\mu,\nu):=\inf_{\pi} \int_{\X\times \X}\di(x,y)^2 \, \de\pi(x,y),
\end{equation}
where the infimum is taken over all \textit{admissible plans}, that is over all $\pi\in \Prob(\X \times \X)$ such that $(\p_1)_\# \pi= \mu$ and $(\p_2)_\# \pi= \nu$. Here $\p_i: \X \times \X \to \X$ denotes the projection on the $i$-th factor.
The infimum in \eqref{eq:W2} is always attained, the admissible plans realizing it are called \textit{optimal transport plans} and the set that contains all of them is denoted by $\opt(\mu,\nu)$. The optimality of a transport plan can be equivalently characterized with the notion of $\di^2$-cyclical monotonicity (see \cite[Definizione 3.10]{MR4294651} for the definition). In particular, an admissible plan $\pi\in \Prob(\X \times \X)$ is optimal between its marginals if and only if it is concentrated on a $\di^2$-cyclically monotone ($\sigma$-compact) set $\Gamma\in\X\times\X$, cf. \cite[Theorem 4.2]{MR4294651}.   

It is well known that $W_2$ is a complete and separable distance on $\Prob_2(\X)$. Moreover, the convergence with respect to the Wasserstein distance is characterized in the following way, cf. \cite[Theorem~8.8]{MR4294651}:
\begin{equation*}
    \mu_n \overset{W_2}{\longrightarrow}\mu \quad \iff \quad \mu_n \rightharpoonup \mu\, \text{ and } \int \di(x_0,x)^2 \de \mu_n \to \int \di(x_0,x)^2 \de \mu \quad \forall x_0\in \X.
\end{equation*}
In the last formula the symbol $\rightharpoonup$ denotes the weak
convergence of measures, i.e. the one in duality with the space of continuous and bounded functions $C_b(\X)$. This description shows in particular that, for a sequence of probability measures with uniformly bounded support, the $W_2$-convergence is equivalent to the weak one. Moreover, as a consequence of Riesz and Banach--Alaoglu theorems, if $(\X,\di)$ is compact then $\Prob(\X)= \Prob_2(\X)$ is compact as well, with respect to the Wasserstein distance $W_2$. For the same reason, every family of measures in $\Prob_2(\X)$, having the same compact support, will be $W_2$-precompact (even if $\X$ is not compact).

Let $C([0, 1], \X)$ be the set of continuous functions from $[0,1]$ to $\X$ and define the $t$-\textit{evaluation map} as $e_t:C([0,1],\X) \to \X$; $e_t(\gamma):= \gamma(t)$, for $\gamma \in C([0,1],\X)$, for any $t\in [0,1]$. A curve $\gamma\in C([0, 1], \X)$ is called \textit{geodesic} if 
\begin{equation}
    \di(\gamma(s), \gamma(t)) = |t-s| \cdot \di(\gamma(0), \gamma(1)) \quad \text{for every }s,t\in[0,1]\,.
\end{equation}
We will denote by $\Geo(\X)$ the space of constant speed geodesics in $(\X,\di)$ parametrized on $[0,1]$. The metric space $(\X,\di)$ is said to be \emph{geodesic} if every pair of points is connected by a geodesic,  Notice that every measure $\eta\in\Prob(C([0,1],\X))$ induces the curve  $[0,1]\ni t \mapsto \mu_t=(e_t)_\# \eta$ in the space of probability measures $\Prob(\X)$. If $(\X,\di)$ is a geodesic metric space then $(\Prob_2(\X),W_2)$ is a geodesic metric space as well. More precisely, given two measures $\mu,\nu\in \Prob_2(\X)$, the curve $\{\mu_t\}_{t\in [0,1]} \subset \Prob_2(\X)$ connecting $\mu$ and $\nu$ is a Wasserstein geodesic if and only if there exists $\eta\in\Prob(C([0,1],\X))$ inducing $\{\mu_t\}_{t\in [0,1]}$ (that is $\mu_t=(e_t)_\# \eta$ for every $t\in [0,1]$), which is concentrated on $\Geo(\X)$ and satisfies $(e_0,e_1)_\#\eta\in\opt(\mu,\nu)$. In this case it is said that $\eta$ is an \textit{optimal geodesic plan} between $\mu$ and $\nu$ and this will be denoted as $\eta\in\optgeo(\mu,\nu)$.

\begin{defn}[Essentially non-branching metric space]
In a metric space $(\X,\di)$, a subset $G\subset \Geo(\X)$ is called \emph{non-branching} if for any pair of geodesics $\gamma_1,\gamma_2\in G$ such that $\gamma_1\neq\gamma_2$, it holds that 
 \begin{equation}
     \rest{0}{t} \gamma_1 \neq  \rest{0}{t} \gamma_1 \qquad \text{for every }t\in (0,1).
 \end{equation}
 A metric measure space $(\X,\mathsf{d},\mathfrak{m})$ is said to be \textit{essentially non-branching} if for every absolutely continuous measures $\mu_0,\mu_1\in\Prob^{ac}(\X,\m)$, every optimal geodesic plan $\eta$ connecting them is concentrated on a non-branching set of geodesics.
\end{defn}


\begin{defn}[Midpoint set]
Let $(X,\di)$ be a metric space. For every $t\in [0,1]$ and any pair of sets $A,B\subset \X$ we define the set of $t$-midpoints between $A$ and $B$ as
\begin{equation*}
    M_t(A,B):= e_t \big( \{\gamma \in \Geo(\X) \, :\, \gamma(0) \in A, \, \gamma(1) \in B\}\big)
\end{equation*}
 We will adopt the notation $M_t(x,A):=M_t(\{x\},A)$ and $M_t(A,x):=M_t(A,\{x\})$ for every $x\in \X$ and $A\subset \X$. 
 \end{defn}

Observe that in general the set $M_t(A,B)$ is not Borel measurable, even if the sets $A$ and $B$ are Borel. In the following we will need to evaluate the measure of $M_t(A,B)$, for this reason we introduce the measure $\bar \m$ as the outer measure associated to $\m$. This measure will not play a significant role and will only be used when dealing with sets of the form $M_t(A,B)$. In particular, having a control on the measure of some suitable sets of $t$-midpoints, is sufficient to deduce some nice properties regarding the structure of $W_2$-geodesics, as shown in Proposition \ref{prop:Kell}.

\begin{defn}{\cite[Definition 5.1]{MR3709114}}
\label{def:qualitatively_non_deg}
A measure $\m$ on a metric space $(\X,\di)$ is said to be \textit{qualitatively non-degenerate} if for every $R>0$ and $\bar x \in \X$ there exists a function $f_{R, \bar x}:(0,1)\to(0,\infty)$ with
\begin{equation*}
    \limsup _{t \to 0} f_{R, \bar x}(t)>\frac{1}{2},
\end{equation*}
such that for every $x\in \mathsf B_R(\bar x)$ and every Borel subset $A \subset \mathsf B_R(\bar x)$ 
\begin{equation*}
    \bar \m(M_t(A,x)) \geq f_{R,\bar x}(t) \cdot  \m(A).
\end{equation*}
\end{defn}


\begin{defn}{\cite[Definition 3.1]{MR3709114}}
\label{def:gtb_sip}
We say that a metric measure space $(\X,\di,\m)$ has the \emph{good transport behaviour} if, for
every pair $\mu_0,\mu_1\in \Prob_2(\X)$ with $\mu_0\ll\m$, any optimal transport plan between $\mu_0$ and $\mu_1$ is induced by a map. We say that $(\X,\di,\m)$ has the \emph{strong interpolation property} if for
every pair $\mu_0,\mu_1\in \Prob_2(\X)$ with $\mu_0\ll\m$, there exists a unique optimal geodesic plan $\eta \in \optgeo(\mu_0,\mu_1)$ and is induced by a map and such that $(e_t)_\# \eta \ll \m$ for every $t\in [0,1)$.
\end{defn}

Recall that a metric space $(\X, \di)$ is said to be \textit{proper} if every closed and bounded set is compact.

\begin{prop}{\cite[Theorem 5.8, Corollary 5.9]{MR3709114}}\label{prop:Kell}
Assume $(\X,\di,\m)$ is a proper, geodesic, essentially non-branching metric measure space and $\m$ is qualitatively non-degenerate. Then, $(\X,\di,\m)$ has both the good transport behaviour and the strong interpolation property.
\end{prop}

\section{The \texorpdfstring{$\cd$}{CD} condition and the Brunn-Minkowski inequality}
\label{sec:cd_cond_bruno}

Before going through the definition of the $\cd$ condition, we introduce the two object which are necessary to do it: the distortion coefficient and the R\'enyi entropy functional.
For every $K \in \R$ and $N >0$ we define 
\begin{equation}
    \sigma_{K,N}^{(t)} (\theta) := \left\{\begin{array}{ll}+\infty  &  N \pi^2\leq K \theta^2 \\ \frac{\sin \left(t \theta \sqrt{\frac{K}{N}}\right)}{\sin \left(\theta \sqrt{\frac{K}{N}}\right)} & 0<K \theta^2 < N \pi^2\\  t & K=0 \text { or } N=\infty \\ \frac{\sinh \left(t \theta \sqrt{\frac{-K}{N}}\right)}{\sinh \left(\theta \sqrt{\frac{-K}{N}}\right)} & K<0\end{array}\right. ,
\end{equation}
 while for $K \in \R$ and $N >1$ we introduce 
 \begin{equation}
     \tau_{K, N}^{(t)}(\theta):=t^{\frac{1}{N}} \sigma_{K, N-1}^{(t)}(\theta)^{1-\frac{1}{N}}.
 \end{equation}
 These coefficients have nice monotonicity properties, in particular for every fixed $t\in (0,1)$, $K \in \R$ and $N>1$ 
 \begin{equation}\label{eq:monotonicitytau}
     \text{the function }\,\,\theta \mapsto  \tau_{K, N}^{(t)}(\theta)  \,\,\text{ is } \left.\begin{array}{ll}\text{nondecreasing} &  \text{ if }K \geq 0 \\ \text{nonincreasing} & \text{ if }K<0\end{array}\right. .
 \end{equation}
The $N$-R\'enyi entropy functional on $\Prob_2(\X)$ is defined as 
\begin{equation}
    \mathcal{E}_N (\mu)=  -\int_{\X} \rho(x)^{1-\frac{1}{N}} \, \de\m(x) \qquad \forall \mu \in \Prob_2(\X)\,,
\end{equation}
 where $\rho$ is the density of the absolutely continuous part of $\mu$, with respect to $\m$. It is well known (see for instance \cite[Lemma 4.1]{MR2237206}) that the $N$-R\'enyi entropy is lower semicontinuous in $\big(\Prob_2(\X), W_2\big)$, if the reference measure $\m$ has finite total mass. In general, for metric measure spaces with possibly infinite mass, the lower semicontinuity holds for $W_2$-converging sequences of measures concentrated on the same bounded set.
 
\begin{defn}[$\cd(K,N)$ condition]
\label{def:cd_condition}
Given $K \in \R$ and $N >1$, a metric measure space $(\X,\di,\m)$ is said to satisfy the \emph{curvature dimension} condition $\cd(K,N)$ (or simply to be a $\cd(K,N)$ space) if for every pair of measures $\mu_0=\rho_0\m,\mu_1= \rho_1 \m \in \Prob^{ac}(\X,\m)$, there exists a $W_2$-geodesic $\eta \in \Prob(\Geo(\X))$ connecting them, such that $(e_t)_\# \eta=: \mu_t =\rho_t \m \ll \m$, for every $t\in [0,1]$, and the following inequality holds for every $N'\geq N$ and every $t \in [0,1]$:
\begin{equation}\label{eq:CDcond}
    \E_{N'}(\mu_t) \leq - \int_{\X \times \X} \Big[ \tau^{(1-t)}_{K,N'} \big(\di(x,y) \big) \rho_{0}(x)^{-\frac{1}{N'}} +    \tau^{(t)}_{K,N'} \big(\di(x,y) \big) \rho_{1}(y)^{-\frac{1}{N'}} \Big]   \de \pi(x,y),
\end{equation}
where $\pi= (e_0,e_1)_\# \eta \in \opt(\mu_0,\mu_1)$. 
\end{defn}

\begin{notation}
In the following, in order to ease the notation we will sometimes denote by $T^{(t)}_{K,N'}(\pi|\m)$ the right hand side of \eqref{eq:CDcond}, that is 
\begin{equation}
    T^{(t)}_{K,N'}(\pi|\m)= - \int_{\X \times \X} \Big[ \tau^{(1-t)}_{K,N'} \big(\di(x,y) \big) \rho_{0}(x)^{-\frac{1}{N'}} +    \tau^{(t)}_{K,N'} \big(\di(x,y) \big) \rho_{1}(y)^{-\frac{1}{N'}} \Big]   \de \pi(x,y).
\end{equation}
Notice that it is not necessary to explicit the dependence of the integral on the densities $\rho_0$ and $\rho_1$, because this information is already encoded in $\pi$ and $\m$, in fact $(\p_1)_\# \pi = \mu_0= \rho_0 \m$ and $(\p_2)_\# \pi = \mu_1= \rho_1 \m$.
\end{notation} 

We now want to state a sufficient criterion to verify the $\cd(K,N)$ condition, allowing to test the definition only on suitable pairs of marginals. To this aim, we introduce the notion of bounded probability measure.

\begin{defn}[Bounded probability measure]
\label{def:bounded_meas}
A probability measure $\mu \in \Prob^{ac}(\X,\m)$ is said to be \emph{bounded} if it has bounded support and density bounded from above and below away from zero. A subset $A \subset \Prob^{ac}(\X,\m)$ is said to be \emph{uniformly bounded} if there exist a bounded set $K$ and two constants $C>c>0$ such that for every $\mu= \rho \m \in A$, $\spt (\mu) =K$ and $c\leq \rho \leq C$ $\m$-almost everywhere on $K$.
\end{defn}

\begin{notation}
Given a Borel set $A \subset \X$ such that $0 < \m(A) <\infty$, we will denote by $\m_A$ the normalized restriction of the reference measure to the set $A$, that is 
\begin{equation}
\label{eq:notation_measure}
    \m_A = \frac{\m|_A}{\m(A)}.
\end{equation}
\end{notation}

\begin{prop}\label{prop:toprovecd}
Let $(\X,\di,\m)$ be a proper, geodesic and essentially non-branching metric measure space and assume $\m$ is qualitatively non-degenerate. Then, $(\X,\di,\m)$ satisfies the $\cd(K,N)$ condition if and only if the requirements of Definition \ref{def:cd_condition} hold for any pair of bounded probability measures. 
\end{prop}

\begin{proof}
Suppose that the $\cd(K,N)$ condition holds for every pair of bounded marginals and fix $\mu_0=\rho_0\m,\,\mu_1=\rho_1\m\in \Prob^{ac}(\X,\m)$. According to Proposition \ref{prop:Kell}, both the good transport behaviour and the strong interpolation property hold for $(\X,\di,\m)$. Then, let $\eta$ be the unique optimal geodesic plan connecting $\mu_0$ and $\mu_1$ and $\pi=(e_0,e_1)_{\#}\eta$ the unique optimal transport plan between them. Moreover, $(e_t)_{\#}\eta \ll \m$ for any $t\in [0,1]$ and we denote by $\rho_t$ its density. Fix $x_0\in \X$ and define (up to null sets) the following sets
\begin{equation*}
    A_n:=\{x\in \mathsf{B}_n(x_0)\,:\, 1/n \leq \rho_0(x) \leq n \} \quad \text{and}\quad B_n:=\{x\in \mathsf{B}_n(x_0)\,:\, 1/n \leq \rho_1(x) \leq n \}.
\end{equation*}
Then, we introduce the set 
\begin{equation*}
G_n:= \{ \gamma \in \Geo(\X) \,:\, \gamma(0)\in A_n, \, \gamma(1)\in B_n\}
\end{equation*}
and the measures 
\begin{equation}
    \eta_n := \eta_{G_n}\in\Prob(\Geo(\X)) \quad \text{and} \quad \pi_n := (e_0,e_1)_{\#} \eta_n\in\Prob(\X\times \X) .
\end{equation}
Note that, since $\eta\in\optgeo(\mu_0,\mu_1)$, its restriction $\eta_n$ is still optimal between $\mu_0^n:= (e_0)_{\#} \eta_n$ and $\mu_1^n:= (e_1)_{\#} \eta_n$ and, as a consequence, $\pi_n\in\opt(\mu_0^n,\mu_1^n)$. In addition, according to Proposition \ref{prop:Kell}, $\eta_n$ is the unique optimal geodesic plan between $\mu_0^n$ and $\mu_1^n$, and $\pi_n$ is the unique optimal plan.
Observe that, thanks to the good transport behaviour, both $\pi$ and $\pi^{-1}:=(e_1,e_0)_{\#}\eta$ are induced by a map, thus, defining $\tilde A_n:= e_0(G_n\cap \spt(\eta))\subset A_n$ and $\tilde B_n:= e_1(G_n\cap \spt(\eta))\subset B_n$, it holds that
\begin{equation}
\label{eq:approx_marginals}
    \mu_0^n = \frac{\mu_0|_{\tilde A_n}}{\eta(G_n)} \quad \text{and} \quad \mu_1^n= \frac{\mu_1|_{\tilde B_n}}{\eta(G_n)}.
\end{equation}
This shows in particular that $\mu_0^n$ and $\mu_1^n$ are bounded for every $n$. Moreover, by definition $\eta(G_n)\cdot\eta_n\leq \eta$, thus denoting by $\rho_t^n$ the density of $(e_t)_{\#}\eta_n$ (with respect to $\m$) and setting $\tilde\rho_t^n:=\eta(G_n)\cdot\rho_t^n$ for every $t\in[0,1]$, we have
\begin{equation}
\label{eq:ineq_rotti}
    \tilde \rho_t^n \leq \rho_t \qquad\m\text{-a.e.},\ \forall \, t\in [0,1].
\end{equation}
On the other hand, the families $\{A_n\}_{n\in\N}$ and $\{B_n\}_{n\in\N}$ exhaust the supports of $\mu_0$ and $\mu_1$ respectively, hence $\eta(G_n)\to 1$ and $\pi(\X \times \X \setminus (\tilde A_n \times \tilde B_n))\to 0$ as $n\to \infty$. Applying the $\cd(K,N)$ condition for the bounded marginals $\mu_0^n$ and $\mu_1^n$, we have, for every $N'\geq N$ and for every $t\in [0,1]$, 
\begin{equation}
\begin{split}
     (\eta&(G_n))^{\frac{1}{N'}-1} \int (\tilde\rho_t^n)^{1-\frac 1 {N'}} \de \m =\int (\rho_t^n)^{1-\frac 1 {N'}} \de \m \\
     &\geq \int_{\X \times \X} \Big[ \tau^{(1-t)}_{K,N'} \big(\di(x,y) \big) \rho_{0}^n(x)^{-\frac{1}{N'}} +    \tau^{(t)}_{K,N'} \big(\di(x,y) \big) \rho_{1}^n(y)^{-\frac{1}{N'}} \Big]   \de \pi_n(x,y) \\
     &= (\eta(G_n))^{\frac{1}{N'}-1} \int_{\tilde A_n \times \tilde B_n} \Big[ \tau^{(1-t)}_{K,N'} \big(\di(x,y) \big) \rho_{0}(x)^{-\frac{1}{N'}} +    \tau^{(t)}_{K,N'} \big(\di(x,y) \big) \rho_{1}(y)^{-\frac{1}{N'}} \Big]  \de \pi(x,y),
\end{split}
\end{equation}
where the last equality follows from \eqref{eq:approx_marginals} and from the fact that $\pi_{\tilde A_n\times \tilde B_n}$ coincides with $\pi_n$. 
Simplifying the term $\eta(G_n)$ (which is definitely strictly greater than $0$) and using \eqref{eq:ineq_rotti}, we obtain for every $N'\geq N$ and for every $t\in [0,1]$
\begin{equation}
\begin{split}
     \int \rho_t^{1-\frac{1}{N'}} \de \m &\geq \int (\tilde\rho_t^n)^{1-\frac 1 {N'}} \de \m \\
     &\geq \int_{\tilde A_n \times \tilde B_n} \Big[ \tau^{(1-t)}_{K,N'} \big(\di(x,y) \big) \rho_{0}(x)^{-\frac{1}{N'}} +    \tau^{(t)}_{K,N'} \big(\di(x,y) \big) \rho_{1}(y)^{-\frac{1}{N'}} \Big]  \de \pi(x,y),
\end{split}
\end{equation}
and taking the limit as $n\to \infty$, we conclude that 
\begin{equation}
     \int \rho_t^{1-\frac{1}{N'}} \de \m \geq \int_{\X \times \X} \Big[ \tau^{(1-t)}_{K,N'} \big(\di(x,y) \big) \rho_{0}(x)^{-\frac{1}{N'}} +    \tau^{(t)}_{K,N'} \big(\di(x,y) \big) \rho_{1}(y)^{-\frac{1}{N'}} \Big]  \de \pi(x,y),
\end{equation}
which is exactly \eqref{eq:CDcond} for $\mu_0$ and $\mu_1$. This concludes the proof. 
\end{proof}

We introduce now a generalized version of the classical Brunn--Minkowski inequality to the non-smooth setting. Similarly to $\cd$ condition, this inequality takes into account dimensional and curvature parameters.

\begin{defn}[Brunn--Minkowski inequality] 
\label{def:BRUNOOO}
Let $(\X,\di,\m)$ be a metric measure space and let $K \in \R$ and $N >1$. We say that $(\X,\di,\m)$ supports the \emph{Brunn--Minkowski inequality} $\bm(K,N)$ if, for every pair of nonempty Borel sets $A,B \subset \spt(\m)$, the following inequality holds for every $N'\geq N$ and every $t \in [0,1]$:
\begin{equation}\label{eq:BM}
    \bar \m \big(M_t(A,B)\big) \big)^ \frac{1}{N'} \geq \tau_{K,N'}^{(1-t)} \big(\Theta (A,B)\big) \cdot \m(A)^ \frac{1}{N'} + \tau_{K,N'}^{(t)} \big(\Theta (A,B)\big) \cdot \m(B)^ \frac{1}{N'}, 
\end{equation}
where 
\begin{equation}\label{eq:defTheta}
    \Theta (A,B):=\left\{\begin{array}{ll}\displaystyle{\inf_{x \in A,\, y \in B} }\di(x, y) & \text { if } K \geq 0, \\ \displaystyle{\sup _{x \in A,\, y \in B} }\di(x, y) & \text { if } K<0.\end{array}\right. 
\end{equation}
\end{defn} 
Similarly as for the $t$-midpoints, we adopt the notation $\Theta(A,x):= \Theta(A,\{ x \})$ and $\Theta(x,A) := \Theta(\{ x\},A)$, for every $x \in \X$ and $A \subset \X$.

\begin{lem}
\label{lem:bm_proper_geod_etc}
Let $(\X,\di,\m)$ be a metric measure space supporting $\bm(K,N)$. Then, $(\spt(\m),\di)$ is a Polish, geodesic and proper metric space. Moreover, $\m$ is a Radon measure.    
\end{lem}

\begin{proof}
Since $(\X,\di)$ is Polish and $\spt(\m)$ is closed, the metric space $(\spt(\m),\di)$ is Polish as well. Moreover, from the proof of \cite[Theorem 2.3]{MR2237207}, $\bm(K,N)$ implies that $(\spt(\m),\di,\m)$ satisfies a Bishop--Gromov inequality, and thus $\m$ is a doubling measure. By a standard argument, this means that $(\spt(\m),\di)$ is a doubling metric space, i.e.\ any bounded set is totally bounded, therefore it is proper and also $\sigma$-compact. As a consequence, $\m$ is Radon, being a locally finite measure on a locally compact, second countable space (see for example \cite[Thoerem 7.8]{MR1681462}). We prove now that $(\spt(\m),\di)$ is length: let $x,y\in \spt(\m)$, $\eps>0$ and fix $A_\eps:=\mathsf{B}_\eps(x)\cap\spt(\m)$ and $B_\eps:=\mathsf{B}_\eps(y)\cap\spt(\m)$. Applying $\bm(K,N)$ we deduce that $\bar \m(M_{1/2}(A_\eps,B_\eps))>0$, therefore there exists $z\in M_{1/2}(A_\eps,B_\eps)\cap\spt(\m)$. In particular, by construction, this implies that:
\begin{equation}
    \di(x,z),\di(z,y)\leq \frac12\di( x, y)+\eps.
\end{equation}
Since $x$, $y$ and $\eps$ are arbitrary, we can conclude that $(\spt(\m),\di)$ is a length space (see \cite[Proposition 1.4]{MR1377265}). Finally, a complete, proper and length space is geodesic. 
\end{proof}

Applying the previous lemma, we can deduce that a metric measure space supporting the Brunn--Minkowski inequality has the properties of Definition \ref{def:gtb_sip}.

\begin{cor}\label{cor:BMnonfaschifo}
Let $(\X,\di, \m)$ be an essentially non-branching metric measure space supporting the Brunn--Minkowski inequality $\bm(K,N)$. Then, $(\X,\di, \m)$ has the good transport behavior and the strong interpolation property.
\end{cor}

\begin{proof}
First of all, we restrict ourselves to the support of $\m$ and consider the metric measure space $(\spt(\m),\di,\m)$. From Lemma \ref{lem:bm_proper_geod_etc}, this is a proper and geodesic metric space. Second of all, letting $R>0$, $\bar x\in\spt(\m)$ and $A\subset\mathsf{B}_R(\bar x)\cap\spt(\m)$ Borel, we may apply $\bm(K,N)$ to $A$ and $\bar x$, obtaining
\begin{equation}
    \bar\m \big(M_t(A,\bar x) \big)\geq \tau_{K,N}^{(1-t)} (\Theta(A,\bar x))^{N}  \m(A),
\end{equation}
for any $t\in [0,1]$. This shows that $\m$ is qualitatively non-degenerate on its support. Finally, applying Proposition \ref{prop:Kell} to the metric measure space $(\spt(\m),\di,\m)$ we conclude the proof. Note that the good transport behavior and the strong interpolation property are only related to optimal transport, which in turn depends on the metric measure structure of $(\X,\di,\m)$ only on the support of $\m$.
\end{proof}

In this paper, we study a stronger version of the Brunn--Minkowski inequality, which is more sensitive to the optimal transport interpolation. 

\begin{defn}[Strong Brunn--Minkowski inequality]
\label{def:forte_bruno}
Let $(\X,\di,\m)$ be a metric measure space and let $K \in \R$ and $N >1$. We say that $(\X,\di,\m)$ supports the \emph{strong Brunn--Minkowski inequality} $\sbm(K,N)$ if, for every pair of Borel sets $A,B \subset \spt(\m)$ such that $0 < \m(A),\m(B) <\infty$, there exists $\eta \in \optgeo(\m_A,\m_B)$, where $\m_A,\m_B$ are as in \eqref{eq:notation_measure}, such that the following inequality holds for every $N'\geq N$ and every $t \in [0,1]$
\begin{equation}\label{eq:SBM}
    \m \big(\spt \big((e_t)_\# \eta \big) \big)^ \frac{1}{N'} \geq \tau_{K,N'}^{(1-t)} \big(\Theta(A,B)\big) \cdot \m(A)^ \frac{1}{N'} + \tau_{K,N'}^{(t)} \big(\Theta(A,B)\big) \cdot \m(B)^ \frac{1}{N'},
\end{equation}
where $\Theta (A,B)$ is defined in \eqref{eq:defTheta}.
\end{defn} 

\begin{prop}\label{prop:nemmenoSBM}
Given $K \in \R$ and $N >1$, if the metric measure space $(\X,\di,\m)$ supports the strong Brunn--Minkowski inequality $\sbm(K,N)$, it also supports the Brunn--Minkowski inequality $\bm(K,N)$.
\end{prop}

\begin{proof}
Fix $N'\geq N$ and $t\in [0,1]$. Let $A,B\subset \spt(\m)$ be two Borel sets with $0 < \m(A),\m(B) <\infty$. Then, for every $\eta \in \optgeo(\m_A,\m_B)$ it holds that $\spt \big((e_t)_\# \eta) \subset M_t(A,B)$ up to a $\m$-null set. Therefore, $\sbm(K,N)$ implies \eqref{eq:BM} for sets with finite and positive measure. Moreover, since the proof of Lemma \ref{lem:bm_proper_geod_etc} only relies on $\bm(K,N)$ for sets of finite and positive measure, we can deduce that $(\spt(\m),\di)$ is proper, geodesic and $\m$ is Radon. Let us prove now $\bm(K,N)$ for any compact sets $A,B\subset\spt(\m)$, with possibly zero measure. If $\m(A)=\m(B)=0$, there's nothing to prove, hence we may assume $\m(A)>0$ and $\m(B)=0$. Moreover, for the time being, assume also that $B=\{x\}$, where $x\in\spt(\m)$.
 Applying the Brunn--Minkowski inequality with the sets $A$ and $\mathsf{B}_r(x)$ (for $r>0$) we obtain 
\begin{equation}\label{eq:proofGTB}
    \m \big(M_t(A,\mathsf{B}_r(x)\big) \big)^ \frac{1}{N'} \geq \tau_{K,N'}^{(1-t)} (\Theta_r) \cdot \m(A)^ \frac{1}{N'} + \tau_{K,N'}^{(t)} (\Theta_r) \cdot \m(\mathsf{B}_r(x))^ \frac{1}{N'} ,
\end{equation}
where $\Theta_r:= \Theta\big(A, \mathsf{B}_r(x)\big)$. On the other hand, it can be proven that
\begin{equation}
   \bigcap_{r>0} M_t(A,\mathsf{B}_r(x)) = M_t(A,x) .
\end{equation}
The $\supset$ inclusion is obvious, while to prove $\subset$ we take $w \in \bigcap_{r>0}M_t(A,\mathsf{B}_r(x))$ and we observe that, given a sequence $\{r_n\}_{n\in \N}$ converging to $0$, there exist $a_n\in A$ and $x_n\in \mathsf{B}_{r_n}(x)$ such that $w$ is a $t$-midpoint of $a_n$ and $x_n$. Since $A$ is compact, up to subsequences $a_n \to a\in A$ and then, by Ascoli--Arzel\`a theorem, $w$ is a $t$-midpoint of $a$ and $x$, thus $w\in M_t(A,x)$. At this point, noting that the sets $M_t(A,\mathsf{B}_r(x))$ are decreasing as $r\to 0$, we can pass to the limit \eqref{eq:proofGTB} and obtain
\begin{equation}
\label{eq:bm_point_compact}
    \m \big(M_t(A,x) \big)^ \frac{1}{N'} \geq \tau_{K,N'}^{(1-t)} (\Theta(A, x)) \cdot \m(A)^ \frac{1}{N'} .
\end{equation}
Now, let $B\subset\spt(\m)$ any compact set with $\m(B)=0$. Then for any $x\in B$, we have 
\begin{equation}
\label{eq:monotonicity_midpoints_tau}
    M_t(A,x)\subset M_t(A,B)\qquad\text{and}\qquad\tau_{K,N'}^{(1-t)}(\Theta(A,x))\geq\tau_{K,N'}^{(1-t)}(\Theta(A,B)),
\end{equation}
where the inequality follows from \eqref{eq:monotonicitytau}. Thus, we may apply \eqref{eq:bm_point_compact} and \eqref{eq:monotonicity_midpoints_tau}, obtaining 
\begin{equation}
\label{eq:bm_compact_compact}
\begin{split}
    \m \big(M_t(A,B) \big)^ \frac{1}{N'} &\geq 
    \m \big(M_t(A,x) \big)^ \frac{1}{N'} \\
    &\geq \tau_{K,N'}^{(1-t)} (\Theta(A, x)) \cdot \m(A)^ \frac{1}{N'}\geq \tau_{K,N'}^{(1-t)} (\Theta(A, B)) \cdot \m(A)^ \frac{1}{N'} .
\end{split}
\end{equation}

In order to prove \eqref{eq:BM} for any Borel sets $A,B\subset\spt(\m)$, with possibly zero or infinite measure, we use the inner regularity of $\m$. In particular, there exist two sequences of compact sets $\{A_n\}_{n\in\N}$ and $\{B_n\}_{n\in\N}$ such that 
\begin{equation}
    A_n\subset A,\ \m(A_n)\to\m(A)\qquad\text{and}\qquad B_n\subset B,\ \m(B_n)\to\m(B).
\end{equation}
For the sets $A_n$ and $B_n$, inequality \eqref{eq:bm_compact_compact} holds, therefore, using the monotonicity of the $t$-midpoints set and of the distortion coefficients as in \eqref{eq:monotonicity_midpoints_tau}, we obtain that
\begin{equation}
\label{eq:bm_approx}
     \bar \m \big(M_t(A,B)\big) \big)^ \frac{1}{N'} \geq \tau_{K,N'}^{(1-t)} \big(\Theta (A,B)\big) \cdot \m(A_n)^ \frac{1}{N'} + \tau_{K,N'}^{(t)} \big(\Theta (A,B)\big) \cdot \m(B_n)^ \frac{1}{N'}.
\end{equation}
Passing to the limit the right-hand side of \eqref{eq:bm_approx}, we finally conclude that $\bm(K,N)$ holds also for $A$ and $B$. 
\end{proof}

\begin{rmk}\label{rmk:Dt}
From Corollary \ref{cor:BMnonfaschifo} and the previous proposition, an essentially non-branching metric measure space $(X,\di,\m)$ supporting $\sbm(K,N)$ has the good transport behavior and the strong interpolation property, cf. Definition \ref{def:gtb_sip}. Thus, given $A,B\subset\spt(\m)$ Borel sets with finite and positive measure, there exists a unique $\eta\in\optgeo(\m_A,\m_B)$, depending only on the sets $A$ and $B$. Hence, we can introduce the following notation without ambiguity:
\begin{equation}
    \label{eq:ambiguous_notation?}
    D_t(A,B):=\spt\big((e_t)_\#\eta\big),\qquad\forall\,t\in [0,1].
\end{equation}
In particular, the inequality \eqref{eq:SBM} now reads as follows: 
\begin{equation}
\label{eq:SBM2}
    \m \big(D_t(A,B)\big)^ \frac{1}{N'} \geq \tau_{K,N'}^{(1-t)} \big(\Theta(A,B)\big) \cdot \m(A)^ \frac{1}{N'} + \tau_{K,N'}^{(t)} \big(\Theta(A,B)\big) \cdot \m(B)^ \frac{1}{N'}.
\end{equation}
\end{rmk}

It was already noticed in the first works about $\cd$ spaces (see in particular \cite{MR2237207}) that the $\cd(K,N)$ condition implies the Brunn--Minkowski inequality $\bm(K,N)$. Following the exact same proof is actually possible to deduce that the $\cd(K,N)$ condition implies the strong Brunn--Minkowski inequality $\sbm(K,N)$. In the following we provide a quick proof of this fact, in order to be self-contained and avoid confusion.

\begin{prop}
Let $(\X,\di,\m)$ be a $\cd(K,N)$ space, for some $K \in \R$ and $N >1$. Then, it supports the strong Brunn--Minkowski inequality $\sbm(K,N)$.
\end{prop}

\begin{proof}
Given any pair of Borel sets $A,B \subset \spt(\m)$ such that $0 < \m(A),\m(B) <\infty$, take the optimal geodesic plan $\eta \in \optgeo(\m_A,\m_B)$ satisfying \eqref{eq:CDcond}. In particular, letting $\pi= (e_0,e_1)_\# \eta$, for every $N'\geq N$ it holds that 
\begin{equation}\label{eq:cdtosbm1}
    \begin{split}
        \E_{N'}((e_t)_\# \eta) &\leq - \int_{\X \times \X} \Big[ \tau^{(1-t)}_{K,N'} \big(\di(x,y) \big) \m(A)^{\frac{1}{N'}} +    \tau^{(t)}_{K,N'} \big(\di(x,y) \big)\m(B)^{\frac{1}{N'}}\Big]   \de \pi(x,y) \\
        & \leq - \Big[ \tau_{K,N'}^{(1-t)} (\Theta(A,B)) \cdot \m(A)^ \frac{1}{N'} + \tau_{K,N'}^{(t)} (\Theta(A,B)) \cdot \m(B)^ \frac{1}{N'} \Big],
    \end{split}
\end{equation}
using \eqref{eq:monotonicitytau}. On the other hand, Jensen's inequality ensures that 
\begin{equation}\label{eq:cdtosbm2}
     \E_{N'}((e_t)_\# \eta) = - \int_{\spt ((e_t)_\# \eta ) } \rho_t(x)^{1-\frac{1}{N'}} \de \m(x)
    \geq
     - \m\big(\spt ((e_t)_\# \eta )\big)^\frac{1}{N'},
\end{equation}
where $\rho_t$ denotes the density of $(e_t)_\# \eta$ with respect to $\m$, for every $t\in [0,1]$. Putting together \eqref{eq:cdtosbm1} and \eqref{eq:cdtosbm2}, we obtain \eqref{eq:SBM}, concluding the proof.
\end{proof}

\section{Proof of the main theorem}
\label{sec:main_theorem}
In this section, we prove our main result, Theorem \ref{thm:SBM_imnplies_CD}. For the convenience of the reader, we recall its statement.

\begin{thm}\label{thm:sbmtocd}
Let $(\X,\di,\m)$ be an essentially non-branching metric measure space supporting $\sbm(K,N)$ for some $K \in \R$ and $N >1$. Then, $(\X,\di,\m)$ is a $\cd(K,N)$ space. In particular, $(\X,\di,\m)$ supports $\sbm(K,N)$ if and only if it satisfies $\cd(K,N)$.
\end{thm}

A key idea in our argument is to prove the $\cd(K,N)$ condition for a suitable subclass of bounded probability measures, called step measures. By an approximation strategy, we then extend the result to \emph{all} bounded measures and finally apply Proposition \ref{prop:toprovecd}, to conclude. 

\begin{defn}
We say that a measure $\mu\in \Prob_2(\X)$ is a \emph{step measure} if it can be written as finite sum of measures with constant density with respect to $\m$, that is 
\begin{equation}
    \mu= \sum_{i=1}^N \lambda_i \m_{A_i},
\end{equation}
where, for every $i=1,\dots, N $, $\lambda_i\in\R$ and $A_i$ is a Borel set with $0 < \m(A_i) <\infty$. Moreover, we assume the sets $\{A_i\}_{i=1,\dots, N}$ to be mutually disjoint.
\end{defn}



Note that the entropy $\E_N$ of a measure $\nu\in \Prob^{ac}(\X,\m)$ equals to $\m(\spt(\nu))^{1/N}$ if and only if $\nu$ has constant density. Thus, letting $A,B\subset\spt(\m)$ with finite and positive measure, the $\sbm(K,N)$ inequality would translate directly to an information on the entropy of the $t$-midpoint $\mu_t$ between $\m_A$ and $\m_B$, only if $\mu_t$ had constant density. However, we can not expect this to be true in general. The previous discussion suggests that, in order to promote $\sbm(K,N)$ to an inequality on the entropy, an argument based on a subsequently refined partition of the support of the marginals is needed, built in accordance with the optimal transport coupling. Indeed, using the partition argument, the $t$-midpoint between $\m_A$ and $\m_B$ can be approximated in entropy with a step measure, which by definition has locally constant density. In addition, the $\sbm(K,N)$ inequality, applied to each element of the partition, controls the entropy of the step measure approximant. 
The partition argument works also when replacing the measures $\m_A$ and $\m_B$ with general step measures as marginals. The advantage of proving the $\cd(K,N)$ inequality for the class of step measures is that the latter is sufficiently large to deduce $\cd(K,N)$ for all bounded measures, by approximation.


\begin{thm}\label{thm:cdwithstep}
Let $(\X,\di,\m)$ be an essentially non-branching supporting $\sbm(K,N)$ for some $K \in \R$ and $N >1$ and let $\mu_0,\mu_1\in \Prob^{ac}(\X,\m)$ be two step measures with bounded support. Then, there exists $\eta\in \optgeo(\mu_0,\mu_1)$ such that \eqref{eq:CDcond} holds.
\end{thm}

\begin{proof}
Combining Corollary \ref{cor:BMnonfaschifo} and Proposition \ref{prop:nemmenoSBM}, we deduce that $(\X,\di,\m)$ has the good transport behavior and the strong interpolation property. Therefore, letting $\mu_0$ and $\mu_1$ be step measures, i.e.
\begin{equation}
     \mu_0= \sum_{i=1}^{N_0} \lambda_i^0 \m_{A_i} \quad \text{and} \quad  \mu_1= \sum_{i=1}^{N_1} \lambda_i^1 \m_{B_i},
\end{equation}
there exists $\eta\in \optgeo(\mu_0,\mu_1)$, the unique optimal geodesic plan connecting $\mu_0$ and $\mu_1$. Then, $\pi :=(e_0,e_1)_\# \eta\in \opt(\mu_0,\mu_1) $ is the unique optimal transport plan between $\mu_0$ and $\mu_1$ and it is induced by a map $T$, that is $\pi=(\text{id}, T)_\# \mu_0$.
Suppose for now that $T$ is continuous, we will get rid of this assumption later in the proof. In this case, since $\spt(\mu_0)$ is compact (see Lemma \ref{lem:bm_proper_geod_etc}), for every $\varepsilon>0$, it is possible to find a finite partition in Borel sets $\{P_j^\varepsilon\}_{j=1,\ldots,L_\varepsilon}$ of it, that is $\cup_{j=1}^{L_\varepsilon} P_j^\varepsilon= \spt(\mu_0)$ up to an $\m$-null set and $P_i^\varepsilon \cap P_j^\varepsilon= \emptyset$ if $i\ne j$, with the following properties:
\begin{enumerate}
    \item[(i)] $\m(P_j^\varepsilon) >0$, for every $j= 1,\dots, L_\varepsilon$,
    \item[(ii)] $\diam \big(P_j^\varepsilon\big) <\varepsilon$ and $\diam \big(T(P_j^\varepsilon) \big)<\varepsilon$, for every $j= 1,\dots, L_\varepsilon$,
    \item[(iii)] for every $j= 1,\dots, L_\varepsilon$, there exists $i(j)$ such that $P_j^\varepsilon \subset A_{i(j)}$ ,
    \item[(iv)] for every $j= 1,\dots, L_\varepsilon$, there exists $\iota(j)$ such that $T(P_j^\varepsilon) \subset B_{\iota(j)}$.
\end{enumerate}
For example, consider the sets $\{P_{i,j}\}_{i,j}$ defined as $P_{i,j}:= A_i \cap T^{-1}(B_j)$, which already satisfy the properties (iii) and (iv).
The sought partition can then be found as a suitable refinement of the partition $\{P_{i,j}\}_{i,j}$, ensuring the property (ii) using the equicontinuity of the map $T$ on the compact set $\spt(\mu_0)$, and condition (i) by neglecting the sets with zero $\m$-measure.

We observe that the good transport behavior implies that the unique optimal map $T^{-1}$ from $\mu_1$ to $\mu_0$ is such that $T^{-1} \circ T = \text{id}$ $\mu_0$-almost everywhere. In particular, 
\begin{equation}\label{eq:equalityofmass}
    \mu_0(P_j^\varepsilon) = \mu_0 (T^{-1}\circ T (P_j^\varepsilon)) = T_\# \mu_0 (T (P_j^\varepsilon))= \mu_1 (T (P_j^\varepsilon)).
\end{equation}
Now we define the measures $\mu_0^{\varepsilon,j}$, $\mu_1^{\varepsilon,j} \in \sM_+(\X)$ as
\begin{equation}
    \mu_0^{\varepsilon,j}:= \mu_0(P_j^\varepsilon) \cdot \m_{P_j^\varepsilon}   \quad \text{and} \quad \mu_1^{\varepsilon,j}:= \mu_1(T(P_j^\varepsilon))\cdot \m_{T(P_j^\varepsilon)}=\mu_0(P_j^\varepsilon)\cdot  \m_{T(P_j^\varepsilon)}  \, , 
\end{equation}
where the last equality follows from \eqref{eq:equalityofmass}.
%
Property (iii) of the partition ensures that $\mu_0^{\varepsilon,j}$ and $\mu_0|_{P_j^\varepsilon}$ are both measures of constant density with respect to $\m$ and with equal mass. Therefore $\mu_0^{\varepsilon,j} = \mu_0|_{P_j^\varepsilon}$, and,  as a consequence 
\begin{equation}
     \mu_0=\sum_{j=1}^{L_\varepsilon} \mu_0|_{P_j^\varepsilon}= \sum_{j=1}^{L_\varepsilon} \mu_0^{\varepsilon,j}.
\end{equation}
Similarly, by \eqref{eq:equalityofmass} and property (iv) we conclude that $T_\# \mu_0^{\varepsilon,j}= \mu_1|_{T(P_j^\varepsilon)}= \mu_1^{\eps,j}$, hence
\begin{equation}
   \mu_1= \sum_{j=1}^{L_\eps} \mu_1^{\varepsilon,j}.
\end{equation}
Defining $\eta_j^\varepsilon := \eta|_{e_0^{-1}(P_j^\varepsilon)} \in \sM_+(\Geo(\X))$, it holds that 
$
    \eta = \sum_{j=1}^{L_\varepsilon} \eta_j^\varepsilon
$.
Note that $\Bar{\eta}_j^\varepsilon:= \frac{\eta_j^\varepsilon}{\mu_0(P_j^\varepsilon)} \in \Prob(\Geo(\X))$. Moreover, since it holds that $(e_0,e_1)_\# \eta = \pi=(\text{id}, T)_\# \mu_0$, by \eqref{eq:equalityofmass} we deduce that, for every $j=1,\dots,L_\varepsilon$,
\begin{equation} \label{eq:NONDIRE36}
    \{\Bar{\eta}_j^\varepsilon\}
    = \optgeo\left( \frac{(e_0)_\#\eta_j^\varepsilon}{\mu_0(P_j^\varepsilon)}, \frac{(e_1)_\#\eta_j^\varepsilon}{\mu_0(P_j^\varepsilon)}\right)= \optgeo(\m_{P_j^\varepsilon},\m_{T(P_j^\varepsilon)}).
\end{equation}
Thus, for every $j$, the curve $t \mapsto \Bar{\mu}_t^{\varepsilon,j}:=(e_t)_\# \Bar{\eta}_j^\varepsilon$ is the unique Wasserstein geodesic connecting $\m_{P_j^\varepsilon}$ and $\m_{T(P_j^\varepsilon)}$, hence 
\begin{equation}
\label{eq:propmubar}
    D_t\big(P_j^\varepsilon, T(P_j^\varepsilon)\big) = \spt(\Bar{\mu}_t^{\varepsilon,j})
        \, , \quad 
    \forall t \in [0,1],
\end{equation}
where the set $D_t(\cdot,\cdot)$ is defined in Remark \ref{rmk:Dt}. As a consequence of the strong interpolation property, for every $j$ and $t$ the measure $\Bar{\mu}_t^{\varepsilon,j}$ is absolutely continuous with respect to $\m$, with density $\Bar{\rho}_t^{\varepsilon,j}$.
 Moreover, by definition
 \begin{align}
    \label{eq:proprhobar}
    \Bar{\rho}_t^{\varepsilon,j}>0 
    \quad \Bar\mu_t^{\eps,j} \text{-almost everywhere on} \  D_t\big(P_j^\varepsilon, T(P_j^\varepsilon)\big).
 \end{align}
In addition, we can apply the strong Brunn--Minkowski inequality $\sbm(K,N)$ and deduce that for every $j=1,\dots,L_\varepsilon$,  $N'\geq N$, and $t\in[0,1]$ it holds that
\begin{equation}\label{eq:BMapplied}
    \m \big(D_t\big(P_j^\varepsilon, T(P_j^\varepsilon)\big)\big)^ \frac{1}{N'} \geq \tau_{K,N'}^{(1-t)} (\Theta_j) \cdot \m\big(P_j^\varepsilon\big)^ \frac{1}{N'} + \tau_{K,N'}^{(t)} (\Theta_j) \cdot \m\big(T(P_j^\varepsilon)\big)^ \frac{1}{N'} ,
\end{equation}
where we use the shorthand notation $\Theta_j:= \Theta(P_j^\eps, T(P_j^\eps))$.

The next goal is to find a suitable approximately $t$-intermediate point $\tilde \mu_t^\eps$ between $\mu_0$ and $\mu_1$. We claim that this can be achieved by considering a family of measures $\Tilde{\mu}_t^{\varepsilon,j} \in \sM_+(\X)$ 
supported on the sets $D_t\big(P_j^\varepsilon, T(P_j^\varepsilon)\big)$ and having constant density, then gluing them together. This would allow us to use \eqref{eq:BMapplied} on each set $P_j^\varepsilon$ and provide a lower bound for the entropy of $\tilde \mu_t^\eps$ .
Precisely, we define, for every $t\in [0,1]$, 
\begin{equation}
    \Tilde{\mu}_t^{\varepsilon,j} := \mu_0(P_j^\varepsilon) \cdot \m_{D_t\big(P_j^\varepsilon, T(P_j^\varepsilon)\big)} \, \text{ for every }j=1,\dots,L_\varepsilon \quad \text{and} \quad \Tilde{\mu}_t^\varepsilon := \sum_{j=1}^{L_\varepsilon} \Tilde{\mu}_t^{\varepsilon,j}.
\end{equation}
Note that, for every $t\in [0,1]$, $ D_t\big(P_j^\varepsilon, T(P_j^\varepsilon)\big)$ has positive measure by \eqref{eq:BMapplied}. Moreover, since $ D_t\big(P_j^\varepsilon, T(P_j^\varepsilon)\big)$ is bounded (being contained in the $t$-midpoints of two bounded sets), it also has finite measure, therefore $\Tilde{\mu}_t^{\varepsilon,j}$ is well defined.

Since $(\X,\di,\m)$ is essentially non-branching, one can prove that $D_t\big(P_j^\varepsilon, T(P_j^\varepsilon)\big) \cap D_t\big(P_i^\varepsilon, T(P_i^\varepsilon)\big)$ is a $\m$-null measure set, whenever $i$ and $j$ are different, see for example \cite[Proposition~2.7]{MagnaboscoRigoni21}.
Let $\tilde \rho_t^\eps$ be the density of $\tilde \mu_t^\eps$ with respect to $\m$, i.e. $\Tilde{\mu}_t^\varepsilon= \Tilde{\rho}_t^\varepsilon  \m \in \Prob(\X)$. Then, for any $t\in [0,1]$ and $N'\geq N$, the entropy $\E_{N'}$ of $\Tilde{\mu}_t^\varepsilon$ is given by
\begin{equation}\label{eq:midpointentropy}
\begin{split}
    \E_{N'}(\Tilde{\mu}_t^\varepsilon) = -\int_\X (\Tilde\rho_t^\varepsilon) ^{1-\frac{1}{N'}} \de \m &= -\sum_{j=1}^{L_\varepsilon}\int_{D_t\big(P_j^\varepsilon, T(P_j^\varepsilon)\big) } (\Tilde\rho_t^\varepsilon) ^{1-\frac{1}{N'}} \de \m \\
    &= - \sum_{j=1}^{L_\varepsilon} \mu_0(P_j^\varepsilon)^{1-\frac{1}{N'}} \m \big(D_t\big(P_j^\varepsilon, T(P_j^\varepsilon)\big) \big)^\frac{1}{N'}
\end{split}
\end{equation}
The combination of \eqref{eq:BMapplied} and \eqref{eq:midpointentropy} gives the following estimate, where $\pi_j:= \pi|_{P_j^\varepsilon \times T(P_j^\varepsilon)}$:
\begin{equation}
    \begin{split}
        \E_{N'}(\Tilde{\mu}_t^\varepsilon) &\leq -\sum_{j=1}^{L_\varepsilon} \mu_0(P_j^\varepsilon)^{1-\frac{1}{N'}} \bigg[\tau_{K,N'}^{(1-t)} (\Theta_j) \cdot \m\big(P_j^\varepsilon\big)^ \frac{1}{N'} + \tau_{K,N'}^{(t)} (\Theta_j) \cdot \m\big(T(P_j^\varepsilon)\big)^ \frac{1}{N'}\bigg]\\
        &= -\sum_{j=1}^{L_\varepsilon} \int \Big[\tau_{K,N'}^{(1-t)} (\Theta_j) \rho_0^{-\frac{1}{N'}}(x) + \tau_{K,N'}^{(t)} (\Theta_j) \rho_1^{-\frac{1}{N'}}(y)\Big]\,  \de \pi_j(x,y)\\
        & \leq -\sum_{j=1}^{L_\varepsilon} \int \Big[\tau_{K,N'}^{(1-t)} (\di(x,y)\mp \varepsilon) \rho_0^{-\frac{1}{N'}}(x) + \tau_{K,N'}^{(t)} (\di(x,y)\mp \varepsilon) \rho_1^{-\frac{1}{N'}}(y)\Big]\,  \de \pi_j(x,y)\\
        \label{eq:entconv}
        & = - \int \Big[\tau_{K,N'}^{(1-t)} (\di(x,y)\mp \varepsilon) \rho_0^{-\frac{1}{N'}}(x) + \tau_{K,N'}^{(t)} (\di(x,y)\mp \varepsilon) \rho_1^{-\frac{1}{N'}}(y)\Big]\,  \de \pi(x,y) .
    \end{split}
\end{equation}
Here the symbol $\mp$ denotes that the estimate holds with the minus if $K\geq 0$ and with the plus if $K<0$.  Notice that the first equality follows by properties (iii) and (iv) of the partition, because $\pi_j$ is concentrated on $P_j^\varepsilon \times T(P_j^\varepsilon)$, whereas the second inequality follows from the monotonicity properties of the distortion coefficients \eqref{eq:monotonicitytau} and the diameter bounds (ii).

Recall the definition of the measure $\bar \mu_t^{\eps,j}$, its density $\bar \rho_t^{\eps,j}$, and their properties, in particular \eqref{eq:propmubar} and \eqref{eq:proprhobar}. For every fixed $s \in [0,1]$, we can then define the measure
\begin{equation}
    \Tilde\eta_j^\varepsilon := \frac{1}{\m\big(D_s\big(P_j^\varepsilon, T(P_j^\varepsilon)\big)\big)}\cdot  \frac{\Bar\eta_j^\varepsilon(\de \gamma)}{\Bar{\rho}_s^{\varepsilon,j}(e_s(\gamma))}
        \in \sM_+(\Geo(\X)).
\end{equation}
By construction $(e_s)_\# \Tilde{\eta}_j^\varepsilon = \m_{D_s\big(P_j^\varepsilon, T(P_j^\varepsilon)\big)} $, which in particular shows that $ \Tilde\eta_j^\varepsilon$ is a probability measure. Moreover, $ \Tilde\eta_j^\varepsilon$ is concentrated on $\Geo(\X)$ and $(e_0,e_1)_\# \Tilde\eta_j^\varepsilon$ is concentrated on the same $\di^2$-cyclically monotone set as $(e_0,e_1)_\#\Bar{\eta}_j^\varepsilon$. For these reasons, we have that
\begin{equation}
    \Tilde\eta_j^\varepsilon \in \optgeo\big((e_0)_\# \Tilde\eta_j,(e_1)_\# \Tilde\eta_j \big)
    \, , \quad \forall j =1, \dots, L_\eps .
\end{equation}
On the other hand, the measures defined by $\nu_0^{\varepsilon,j}:=(e_0)_\# \Tilde\eta_j^\varepsilon$ and $\nu_1^{\varepsilon,j}:=(e_1)_\# \Tilde\eta_j^\varepsilon$ are concentrated on $P_j^\varepsilon$ and $T(P_j^\varepsilon)$ respectively. Hence, recalling that, for every $\mu,\nu \in \Prob_2(\X)$, it holds 
\begin{equation}
\label{eq.diamest}
    W_2 (\mu,\nu) \leq \text{diam}\big( \spt(\mu) \cup \spt(\nu)\big),
\end{equation} 
property (ii) of the partition and the triangle inequality allow us to conclude that
\begin{equation}\label{eq:approxofmarg}
\begin{split}
     W_2\big(\m_{P_j^\varepsilon}, \m_{D_s(P_j^\varepsilon, T(P_j^\varepsilon))}\big) & \leq W_2(\m_{P_j^\varepsilon}, \nu_0^{\varepsilon,j}) + W_2\big(\nu_0^{\varepsilon,j}, \m_{D_s(P_j^\varepsilon, T(P_j^\varepsilon))}\big) \\
    & \leq \varepsilon + s \cdot W_2(\nu_0^{\varepsilon,j}, \nu_1^{\varepsilon,j}) \\
    & \leq   3 \varepsilon + s \cdot W_2(\m_{P_j^\varepsilon}, \m_{T(P_j^\varepsilon)}),
\end{split}
\end{equation}
where we repeatedly used \eqref{eq.diamest}, and analogously 
\begin{equation}
    W_2\big(\m_{D_s(P_j^\varepsilon, T(P_j^\varepsilon))} ,\m_{T(P_j^\varepsilon)}\big) \leq   3 \varepsilon + (1-s) \cdot W_2(\m_{P_j^\varepsilon}, \m_{T(P_j^\varepsilon)}) ,
\end{equation}
for every $j=1, \dots,L_\varepsilon$. On the other hand, recalling that $ \mu_0= \sum_{j=1}^{L_\varepsilon}\mu_0(P_j^\varepsilon) \m_{P_j^\varepsilon} $ and $\Tilde\mu_s^\varepsilon=\sum_{j=1}^{L_\varepsilon} \mu_0(P_j^\varepsilon) \m_{D_s(P_j^\varepsilon, T(P_j^\varepsilon))}$, the convexity of $W_2^2$ gives us \begin{equation*}
    W_2^2(\mu_0,\Tilde\mu_s^\varepsilon) \leq \sum_{j=1}^{L_\varepsilon} \mu_0(P_j^\varepsilon) \cdot  W_2^2 \big(\m_{P_j^\varepsilon}, \m_{D_s(P_j^\varepsilon, T(P_j^\varepsilon))}\big).
\end{equation*}
In addition, thanks to the optimality of the map $T$ and the properties of the partition, cf.\ \eqref{eq:NONDIRE36}, we have that
\begin{align}
    \sum_{j=1}^{L_\varepsilon}\mu_0(P_j^\varepsilon) W_2^2(\m_{P_j^\varepsilon}, \m_{T(P_j^\varepsilon)}) = W_2^2(\mu_0,\mu_1) .
\end{align}
Hence, by summing \eqref{eq:approxofmarg} on $j$, we obtain the estimate
\begin{align}
     W_2^2(\mu_0,\Tilde\mu_s^\varepsilon) &\leq \sum_{j=1}^{L_\varepsilon}\mu_0(P_j^\varepsilon) \bigg( 3 \varepsilon + s \cdot W_2(\m_{P_j^\varepsilon}, \m_{T(P_j^\varepsilon)}) \bigg)^2\\
     \label{eq:tmidpoint1}
     &= 9 \varepsilon^2 + 6\varepsilon s \sum_{j=1}^{L_\varepsilon}\mu_0(P_j^\varepsilon) W_2(\m_{P_j^\varepsilon}, \m_{T(P_j^\varepsilon)}) + s^2 \sum_{j=1}^{L_\varepsilon}\mu_0(P_j^\varepsilon) W_2^2(\m_{P_j^\varepsilon}, \m_{T(P_j^\varepsilon)})\\
     &\leq  9 \varepsilon^2 + 6\varepsilon s D  + s^2 \cdot W_2^2(\mu_0,\mu_1),
\end{align}
where,in the third line, we introduced the quantity $D:= \text{diam}(\spt (\mu_0)\cup \spt(\mu_1))$ and applied \eqref{eq.diamest}.
Analogously, we have
\begin{equation}\label{eq:tmidpoint2}
\begin{split}
     W_2(\Tilde\mu_s^\varepsilon,\mu_1)\leq 9 \varepsilon^2 + 6\varepsilon (1-s) D  + (1-s)^2 \cdot W_2^2(\mu_0,\mu_1).
\end{split}
\end{equation}
Observe that we have proven \eqref{eq:tmidpoint1} and \eqref{eq:tmidpoint2} for every $s\in [0,1]$.

We now pass to the limit as $\varepsilon\to 0$. 
Notice that, since $\mu_0$ and $\mu_1$ have bounded support and the space $(\X,\di,\m)$ is proper by Lemma \ref{lem:bm_proper_geod_etc}, for every $t\in [0,1]$, all the measures in the family $\{\Tilde\mu_t^\varepsilon\}_{\varepsilon>0}$ are concentrated on a common compact set. In particular, for every fixed $t\in [0,1]$, the family $\{\Tilde\mu_t^\varepsilon\}_{\varepsilon>0}$ is $W_2$-precompact (see Section \ref{sec:prelim}). Thus we can find a sequence $\{\varepsilon_m\}_{m\in \N}$ converging to $0$ such that 
\begin{equation}
     \Tilde\mu_t^{\varepsilon_m} \xrightarrow{W_2} \mu_t \in \Prob_2(\X) \quad \text{as }m\to \infty.
\end{equation}
Now we can pass \eqref{eq:tmidpoint1} and \eqref{eq:tmidpoint2} to the limit as $m\to \infty$ and obtain that
\begin{equation}
     W_2(\mu_0,\mu_t)\leq t \cdot W_2(\mu_0,\mu_1) \quad\text{and}\quad W_2(\mu_t,\mu_1)\leq  (1-t) \cdot W_2(\mu_0,\mu_1).
\end{equation}
As a consequence, we deduce that $\mu_t$ is the unique $t$-midpoint between $\mu_0$ and $\mu_1$, and
\begin{equation}
     \Tilde\mu_t^{\varepsilon} \xrightarrow{W_2} \mu_t \in \Prob_2(\X) \quad \text{as }\eps\to 0,
\end{equation}
without extracting a subsequence. Repeating the argument for every $t\in [0,1]$, we deduce that the curve $t\mapsto\mu_t$ is the unique Wasserstein geodesic connecting $\mu_0$ and $\mu_1$. Then, we can pass to the limit \eqref{eq:entconv}, using the monotone convergence theorem for the right-hand side and the lower semicontinuity of $\E_{N'}$ for the left-hand side and obtain, for every $t\in [0,1]$ and $N'\geq N$,
\begin{equation}
     \E_{N'}(\mu_t) \leq - \int_{\X \times \X} \Big[ \tau^{(1-t)}_{K,N'} \big(\di(x,y) \big) \rho_{0}(x)^{-\frac{1}{N'}} +    \tau^{(t)}_{K,N'} \big(\di(x,y) \big) \rho_{1}(y)^{-\frac{1}{N'}} \Big]   \de \pi(x,y),
\end{equation}
which is the desired inequality. Note that we were able to exploit the lower semicontinuity of $\E_{N'}$, because for every $t\in [0,1]$ the measures of the family $\{\tilde \mu_t^\eps\}_{\eps>0} \cup \{\mu_t\}$ are concentrated on a common bounded set. 

So far we have proven that, if the optimal map between two step measures with bounded support is continuous, the unique geodesic connecting them satisfies the entropy convexity inequality \eqref{eq:CDcond}.  We now have to get rid of the assumption on the continuity of the optimal transport map $T$. If $T$ is not continuous, recalling that $\m$ is Radon thanks to Lemma \ref{lem:bm_proper_geod_etc}, we may apply Lusin theorem: for every $\epsilon>0$, there exists a compact set $A_\epsilon\subset \spt(\mu_0)$, such that $T$ is continuous on $A_\epsilon$ and $\mu_0(\spt(\mu_0) \setminus A_\epsilon)< \epsilon$. Then, define the measures 
\begin{equation*}
     \eta^\epsilon := \frac{\eta|_{e_0^{-1}(A_\epsilon)}}{\mu_0(A_\epsilon)} \in \Prob(\Geo(\X))  \quad \text{and}\quad  \mu_0^\epsilon:= (e_0)_\#  \eta,\  \mu_1^\epsilon:= (e_1)_\#  \eta \in \Prob^{ac}(\X,\m).
\end{equation*}
Note also that $ \eta^\epsilon$ is the unique optimal geodesic plan in $\optgeo(\mu_0^\epsilon, \mu_1^\epsilon)$. Moreover, exploiting the good transport behaviour as done in the first part of the proof, we deduce that the set $B_\epsilon:=T(A_\epsilon)\subset \spt(\mu_1)$ is such that $ \mu_1^\epsilon = \mu_1|_{B_\epsilon} / \mu_1(B_\epsilon)$, while by definition is clear that $ \mu_0 = \mu_0|_{A_\epsilon} / \mu_0(A_\epsilon)$. In particular, $\mu_0^\epsilon$ and $\mu_1^\epsilon$ are step measures with bounded support, $T|_{A_\epsilon}$ is the optimal map between them and it is continuous. Therefore \eqref{eq:CDcond} holds for $\mu_0^\epsilon$, $ \mu_1^\epsilon$ and $ \eta ^\epsilon$. Repeating the same argument used in the proof of Proposition \ref{prop:toprovecd}  replacing the sets $A_n$ and $B_n$ with $A_\epsilon$ and $B_\epsilon$, respectively, we can pass to the limit as $\epsilon \to 0$ and deduce that \eqref{eq:CDcond} holds for $\mu_0$, $\mu_1$ and $\eta$. This concludes the proof.
\end{proof}

At this point, we proceed by approximation and use Theorem \ref{thm:cdwithstep} to prove $\cd(K,N)$ for every pair of bounded marginals. The next two lemmas serve for this purpose, providing a suitable sequence of step measures converging to a given bounded measure and an upper semicontinuity result for the functional $T_{K,N}^{(t)}(\cdot|\m)$. 

\begin{lem}\label{lem:approxstepmeasures}
Let $\mu = \rho \m \in \Prob^{ac}(\X,\m)$ be bounded, then there exists a sequence of step measures $\{\mu_n= \rho_n \m\}_{n\in\N}$ $W_2$-convergent to $\mu$, such that $\{\mu_n\}_{n\in\N} \cup \{\mu\}$ is uniformly bounded and $\rho_n^{- 1/N'} \to \rho^{-1/N'}$ in $L^1(\m)$ for every $N'>1$.
\end{lem}

\begin{proof}
Let $K\subset \X$ be the (compact) support of $\mu$ and $c>0$ as in Definition \ref{def:bounded_meas}. Let $\{\tilde \rho_n\}_{n\in \N}$ be a sequence of step functions such that $0 \leq \tilde \rho_n \leq \tilde \rho := \rho-c \chi_K$ and 
\begin{align*}
    \lim_{n \to \infty} \int \tilde \rho_n \de \m = \int \tilde \rho \de \m.
\end{align*}
We then define for every $n \in \N$, the measure $\mu_n := \rho_n \m \in \Prob^{ac}(\X,\m) $, where $\rho_n$ is the step function defined as follows: 
\begin{align*}
    \rho_n := \frac{\tilde \rho_n + c \chi_K} {\| \tilde \rho_n + c\chi_K  \|_{L^1(\m)}}.
\end{align*}
Note that $\| \tilde \rho_n + c \chi_K  \|_{L^1(\m)} \to \| \rho \|_{L^1(\m)} = 1$ and, in particular, we observe that
\begin{align*}
    \int 
    \left|
        \rho_n - \rho
    \right| 
        \de \m
\leq
    \| \tilde \rho_n - \tilde \rho \|_{L^1(\m)}
+
    \bigg( 
        \frac1{\| \tilde \rho_n + c\chi_K  \|_{L^1(\m)}}
            - 1
    \bigg)
    \| \tilde \rho_n + c\chi_K  \|_{L^1(\m)}
    \to 0 
\end{align*}
as $n \to \infty$, thus $\rho_n \to \rho$ in $L^1(\m)$. Moreover, the sequence $\{ \mu_n \}_{n\in \N}$ is by construction uniformly bounded, and in particular for every $N'\in \N$, the sequence $\{ \rho_n^{-1/N'} \}_{n \in \N}$ is uniformly bounded from below and above. Furthermore, since $\rho_n \to \rho$ in $L^1(\m)$, up to a (non-relabeled) subsequence, $\rho_n \to \rho$ pointwise $\m$-almost everywhere (see e.g. \cite[Theorem~3.12]{MR924157}). Trivially, this also shows that $\rho_n^{-1/N'} \to \rho^{-1/N'}$ pointwise $\m$-almost everywhere, hence an application of the dominated convergence theorem implies that $\rho_n^{-1/N'} \to \rho^{-1/N'}$ in $L^1(\m)$ and conclude the proof.
\end{proof}

\begin{lem}\label{lem:TKN}
Let $\mu_0= \rho_0 \m,\mu_1=\rho_1 \m \in \Prob^{ac}(\X,\m)$ be bounded and $\pi$ be the unique optimal transport plan between them. Let $\{\mu_0^n=\rho_0^n\m\}_{n\in \N}$, $\{\mu_1^n= \rho_1^n\m\}_{n\in \N}\subset \Prob^{ac}(\X,\m)$ be the approximating sequences provided by Lemma \ref{lem:approxstepmeasures}. Then, letting $\pi_n$ be the unique optimal transport plan between $\mu_0^n$ and $\mu_1^n$, it holds that 
\begin{equation}
    \limsup_{n\to \infty} T^{(t)}_{K,N'}(\pi_n|\m) \leq T^{(t)}_{K,N'}(\pi|\m),
\end{equation}
for every $K\in \R$, $N'>1$ and $t\in[0,1]$.
\end{lem}

\begin{proof}
We follow a strategy similar to the one developed in \cite[Proposition 4.10]{MRS}.
Recall that the sequence $\{\pi_n\}_{n\in \N}$ weakly converges to $\pi\in\opt(\mu_0,\mu_1)$ (see for example \cite[Theorem 6.8]{MR4294651}). Moreover, we observe that it is sufficient to prove
\begin{equation}\label{eq:lscinproof}
    \liminf_{n\to \infty}\int \tau^{(1-t)}_{K,N'} \big(\di(x,y) \big) \rho_0^n(x)^{-\frac{1}{N'}}    \de \pi_n(x,y) \geq \int \tau^{(1-t)}_{K,N'} \big(\di(x,y) \big) \rho_0(x)^{-\frac{1}{N'}}    \de \pi(x,y),
\end{equation}
since the other term can be treated analogously. The space of continuous and bounded functions $C_b(\X)$ is dense in $L^1(\m)$ (see for example \cite[Theorem 3.14]{MR924157}), thus for every $\varepsilon>0$ we can find $g^\varepsilon\in C_b(\X)$ such that $||\rho_0^{-1/N'} - g^\varepsilon ||_{L^1(\m)}<\varepsilon$. Furthermore, for every $n\in\N$ big enough we have that $||(\rho_0^n)^{-1/N'} - g^\varepsilon ||_{L^1(\m)}<2\varepsilon$. In general, the function $(x,y) \mapsto \tau^{(1-t)}_{K,N'} \big(\di(x,y) \big)$ is continuous but not bounded, for this reason, for every $M >0$ we introduce the function
\begin{equation}
    f_M(x,y)= \tau^{(1-t)}_{K,N'} \big(\di(x,y) \big) \wedge M.
\end{equation}
The function $f_M$ is continuous and bounded above by $M$ and therefore 
\begin{equation}
\begin{split}
    \liminf_{n\to \infty} \int f_M (x,y)& \rho_0^n(x)^{-\frac{1}{N'}}    \de \pi_n(x,y) \\
    &\geq \liminf_{n\to \infty} \int f_M (x,y) g^\varepsilon(x)    \de \pi_n(x,y) - M \int |g^\varepsilon-(\rho_0^n)^{-1/N'} |  \de \mu_0^n\\
    &\geq \liminf_{n\to \infty} \int f_M (x,y) g^\varepsilon(x)    \de \pi_n(x,y) -C M \int |g^\varepsilon-(\rho_0^n)^{-1/N'} |  \de \m \\
    &\geq \liminf_{n\to \infty} \int f_M (x,y) g^\varepsilon(x)    \de \pi_n(x,y) -2 \varepsilon C M  \\
    &= \int f_M (x,y) g^\varepsilon(x)    \de \pi(x,y) -2 \varepsilon C M \\
    &\geq \int f_M(x,y) \rho_0(x)^{-\frac{1}{N'}} \de \pi(x,y)- 3 \varepsilon CM,
\end{split}
\end{equation}
where the equality holds because $(x,y)\mapsto f_M (x,y) g^\varepsilon(x)$ is continuous and bounded and $\pi_n\rightharpoonup \pi$, while the constant $C>0$ represents the uniform upper bound on $\{\rho_0^n\}_{n\in\N} \cup \{\rho_0\}$. Since this last inequality holds for every $\varepsilon>0$, we can conclude that 
\begin{equation}
    \liminf_{n\to \infty} \int f_M (x,y) \rho_0^n(x)^{-\frac{1}{N'}}    \de \pi_n(x,y) \geq \int f_M(x,y) \rho_0(x)^{-\frac{1}{N'}} \de \pi(x,y).
\end{equation}
Taking into account this semicontinuity property we deduce that for every $M>0$
\begin{equation}
\begin{split}
     \liminf_{n\to \infty}\int \tau^{(1-t)}_{K,N'} \big(\di(x,y) \big) \rho_0^n(x)^{-\frac{1}{N'}}    \de \pi_n(x,y) &\geq \liminf_{n\to \infty} \int f_M (x,y) \rho_0^n(x)^{-\frac{1}{N'}}    \de \pi_n(x,y)\\ &\geq \int f_M(x,y) \rho_0(x)^{-\frac{1}{N'}} \de \pi(x,y),
\end{split}     
\end{equation}
taking now the limit as $M \to \infty$, the monotone convergence theorem allows to prove \eqref{eq:lscinproof}, concluding the proof.
\end{proof}

\begin{proof}[Proof of Theorem \ref{thm:sbmtocd}.] According to Proposition \ref{prop:toprovecd}, we can limit ourselves to prove the $\cd(K,N)$ condition for bounded marginals. On the other hand, for every pair of bounded marginals $\mu_0,\mu_1\in \Prob^{ac}(\X,\m)$, there exist two approximating sequences $\{\mu_0^n\}_{n \in \N}$ and $\{\mu_1^n\}_{n \in \N}$ satisfying the requirements of Lemma \ref{lem:approxstepmeasures}. For every $n\in \N$ call $\eta_n$ the unique optimal geodesic plan in $\optgeo(\mu_0^n,\mu_1^n)$, let $\pi_n:= (e_0,e_1)_\# \eta_n$ and $\mu_t^n:=(e_t)_\# \eta_n$. For every $n\in \N$, $\mu_0^n$ and $\mu_1^n$ are step measures with bounded support, thus Theorem \ref{thm:cdwithstep} ensures that 
\begin{equation}\label{eq:CDstepapprox}
    \E_{N'}(\mu_t^n) \leq T_{K,N'}^{(t)} (\pi_n | \m) \qquad \text{for every }N'\geq N \text{ and } t\in [0,1].
\end{equation}
Now we want to pass to the limit as $n \to \infty$. 
Notice that, since the families $\{\mu_0^n\}_{n\in\N} \cup \{\mu_0\}$ and $\{\mu_1^n\}_{n\in\N} \cup \{\mu_1\}$ are uniformly bounded and the space $(\X,\di,\m)$ is proper, for every fixed $t\in [0,1]$, all the measures in the family $\{\mu_t^n\}_{n \in \N}$ are concentrated on the same compact set. In particular, for every fixed $t\in [0,1]$, the family $\{\mu_t^n\}_{n \in \N}$ is $W_2$-precompact. We can then extract a (non-relabeled) subsequence such that 
\begin{equation*}
     \mu_t^n \xrightarrow{W_2} \mu_t \in \Prob_2(\X) \quad \text{as }n\to \infty.
\end{equation*}
For every $n\in\N$, the measure $\mu_t^n$ is a $t$-midpoint between $\mu_0^n$ and $\mu_1^n$, moreover $\mu_0^n \to \mu_0$ and  $\mu_1^n \to \mu_1$ with respect to $W_2$, thus $\mu_t$ is the unique $t$-midpoint between $\mu_0$ and $\mu_1$. We can then pass \eqref{eq:CDstepapprox} to the limit as $n\to \infty$, thanks to the lower semicontinuity of the entropy functional $\E_{N'}$ and to Lemma \ref{lem:TKN}, and obtain
\begin{equation*}
     \E_{N'}(\mu_t) \leq T_{K,N'}^{(t)} (\pi | \m), \qquad \text{for every }N'\geq N \text{ and } t\in [0,1],
\end{equation*}
which is \eqref{eq:CDcond}. This concludes the proof.
\end{proof}

\section{Final comments}
\label{sec:finndelany_grandeacquisto}

\subsection{The $\bm(K,N)$ inequality and the $\mcp(K,N)$ condition}

In this section, we explore the relation between the (strong) Brunn--Minkowski inequality and the so-called measure contraction property. We recall here its definition, firstly introduced in \cite{MR2341840}, and an equivalent characterization proved therein. 

\begin{defn}[$\mcp(K,N)$ condition]
Given $K\in\R$ and $N>1$, a metric measure space $(\X,\di,\m)$ is said to satisfy the \emph{measure contraction property} $\mcp(K,N)$ if for every $x\in\spt(\m)$ and a Borel set $A\subset\X$ with $0<\m(A)<\infty$, there exists $\eta\in\optgeo(\delta_x,\m_A)$ such that, for every $t\in[0,1]$,
\begin{equation}
    \frac{1}{\m(A)}\m\geq(e_t)_\#\Big(\tau_{K,N}^{(t)}\big(\di(\gamma_0,\gamma_1)\big)^N\eta(\de\gamma)\Big).
\end{equation}
\end{defn}

\begin{lem}[{\cite[Lemma 2.3]{MR2341840}}]\label{lem:ohta}
Let $(\X,\di,\m)$ be a metric measure space. Assume that, for every $x\in\spt(\m)$ and $A\subset\X$ a Borel set with $0<\m(A)<\infty$, there exists a measurable selection $\Phi\colon A\rightarrow \Geo(\X)$ satisfying $e_0\circ\Phi\equiv x$ and $e_1\circ\Phi=\mathrm{id}_A$ such that, for any Borel $ A'\subset A$,
\begin{equation}
\label{eq:tj_pantaloncini}
    \m\big(e_t(\Phi(A'))\big)\geq \int_{ A'}\tau_{K,N}^{(t)}(\di(x,y))^N \de\m(y).
\end{equation}
Then, $(\X,\di,\m)$ satisfies the $\mcp(K,N)$ condition.
\end{lem}

\begin{prop}\label{prop:capitantadda}
Let $(\X,\di,\m)$ be an essentially non-branching metric measure space supporting the Brunn--Minkowski inequality $\bm(K,N)$. Then, $(\X,\di,\m)$ has the measure contraction property $\mcp(K,N)$.
\end{prop}

\begin{proof}
Let $x\in\spt(\m)$ and $A\subset\X$ a Borel set with $0<\m(A)<\infty$; we can assume, without loss of generality, that $A\subset\spt(\m)$. Then, applying Corollary \ref{cor:BMnonfaschifo}, and in particular the strong interpolation property to the marginals $\m_A$ and $\delta_x$, there exists a unique geodesic connecting $a$ and $x$, for $\m$-a.e.\ $a\in A$. Thus, we have a well-defined measurable selection $\Phi$ satisfying the requirements of Lemma \ref{lem:ohta} and 
\begin{equation}
    \m\big(M_t(x,A')\big)=\m\big(e_t(\Phi(A'))\big),\qquad\text{for every Borel }A'\subset A.
\end{equation}
Let $\eps>0$ and $A'\subset A$ be a Borel set. Define a partition $\{A_n^\eps\}_{n\in\N}$ of $A'$ as follows:
\begin{equation*}
    A_n^\eps:= A' \cap C_n^\eps\qquad\text{where}\qquad C_n^\eps:= \{y \in \X \,:\, n \eps < \di(y,x)\leq (n+1)\eps\}.
\end{equation*}
Applying the $\bm(K,N)$ inequality for the sets $\{x\}$ and $A_n^\eps$ we obtain that
\begin{equation}\label{eq:jeremymorgan}
    \begin{split}
        \m \big(M_t(x,A_n^\eps)\big) &\geq \tau_{K,N}^{(t)}\big(\Theta(x,A_n^\eps)\big)^N \m (A_n^\eps)\\
        &= \int_{A_n^\eps} \tau_{K,N}^{(t)}\big(\Theta(x,A_n^\eps)\big)^N \de \m(z) \geq \int_{A_n^\eps}  \tau_{K,N}^{(t)}( \di(x,z) \mp \eps)^N \de \m(z).
    \end{split}
\end{equation}
Note that $M_t(x,A_n^\eps)\cap M_t(x,A_m^\eps)\ne \emptyset$ whenever $n \ne m$, since by construction every $z \in M_t(x,A_n^\eps)$ is such that $\di (x,z) \in (t  n \eps, t  (n+1) \eps]$. Therefore, we can sum the inequalities \eqref{eq:jeremymorgan} over all $n\in\N$, obtaining 
\begin{equation}
    \m \big(M_t(x,A')\big) \geq \int_{A'}  \tau_{K,N}^{(t)}( \di(x,z) \mp \eps)^N \de \m(z).
\end{equation}
Passing to the limit as $\eps\to 0$ and using Fatou lemma, we deduce inequality \eqref{eq:tj_pantaloncini} for $A'$. By the arbitrariness of $t\in (0,1]$ and $A'\subset A$, we conclude the proof. 
\end{proof}

Proposition \ref{prop:capitantadda} shows that the $\bm(K,N)$ inequality, intended as a curvature dimension bound, is stronger than the $\mcp(K,N)$ condition. The heuristic reason behind this difference is that, while the $\bm(K,N)$ inequality controls the behavior of the set geodesics joining \emph{any} two sets, $\mcp(K,N)$ controls only the interpolation between a constant density measure and a Dirac delta (which, in the case of essentially non-branching spaces, corresponds to a control on geodesics spreading out from a point). This is also confirmed by the existence of weighted Riemannian manifolds where $\mcp(K,N)$ does not imply $\bm(K,N)$, with the same (sharp) constants. In conclusion, the Brunn--Minkowski inequality $\bm(K,N)$ is closer to the $\cd(K,N)$ condition than the measure contraction property $\mcp(K,N)$ is.

\subsection{Relation between $M_t(A,B)$ and $D_t(A,B)$}
\label{sec:Dt_VS_Mt_emph}

In Theorem \ref{thm:sbmtocd}, we proved that $\sbm(K,N)$ is equivalent to $\cd(K,N)$, for essentially non-branching metric measure spaces. In principle, we would like to improve the equivalence, including $\bm(K,N)$, as shown for weighted Riemannian manifolds in \cite{seminalpaper}. The strategy proposed in the proof of Theorem \ref{thm:cdwithstep} could be adapted to deduce stronger the implication $\bm(K,N)\Rightarrow\cd(K,N)$, if we were able to control the difference between (the measure of) the sets $M_t(A,B)$ and $D_t(A,B)$. In particular, we can not expect them to be equal for any couple of sets, also in elementary examples. Consider for instance the metric measure space $(\R^2,|\cdot|,\mathscr{L}^2)$ and the sets $A,B\subset\R^2$, as in Figure \ref{fig:graffi_di_leone}.
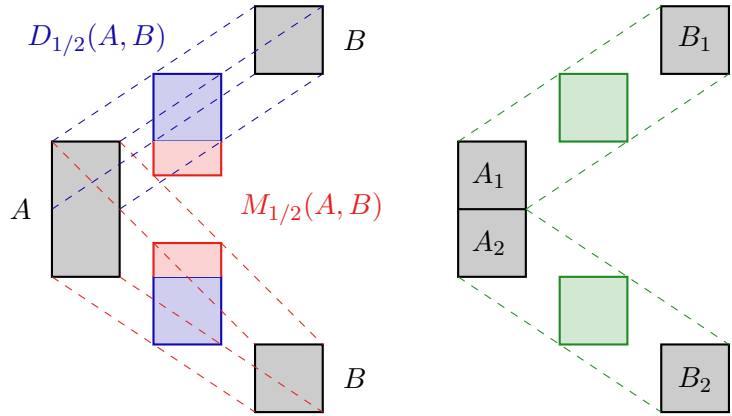
\begin{figure}[h]
\begin{center}
\hspace{-30pt}
\begin{tikzpicture} [scale=0.9]

\filldraw[color=black, fill=black!20,thick](-5,-1)--(-4,-1)--(-4,1)--(-5,1)--cycle;
\filldraw[color=black, fill=black!20,thick](-2,2)--(-1,2)--(-1,3)--(-2,3)--cycle;
\filldraw[color=black, fill=black!20,thick](-2,-2)--(-1,-2)--(-1,-3)--(-2,-3)--cycle;
\filldraw[color=black, fill=black!20,thick](1,-1)--(2,-1)--(2,1)--(1,1)--cycle;
\filldraw[color=black, fill=black!20,thick](4,2)--(5,2)--(5,3)--(4,3)--cycle;
\filldraw[color=black, fill=black!20,thick](4,-2)--(5,-2)--(5,-3)--(4,-3)--cycle;
\filldraw[color=myblue,fill=myblue!20, thick](-3.5,1)--(-2.5,1)--(-2.5,2)--(-3.5,2)--cycle;
\filldraw[color=myblue,fill=myblue!20, thick](-3.5,-1)--(-2.5,-1)--(-2.5,-2)--(-3.5,-2)--cycle;
\filldraw[color=myred,fill=myred!20, thick](-2.5,-1)--(-2.5,-0.5)--(-3.5,-0.5)--(-3.5,-1);
\filldraw[color=myred,fill=myred!20, thick](-2.5,1)--(-2.5,0.5)--(-3.5,0.5)--(-3.5,1);
\draw[dashed, color=myred] (-5,1)--(-2,-2);
\draw[dashed, color=myred] (-4,1)--(-1,-2);
\draw[dashed, color=myred] (-5,-1)--(-2,-3);
\draw[dashed, color=myred] (-4,-1)--(-1,-3);
\draw[dashed, color=myblue] (-5,0)--(-2,2);
\draw[dashed, color=myblue] (-4,0)--(-1,2);
\draw[dashed, color=myblue] (-5,1)--(-2,3);
\draw[dashed, color=myblue] (-4,1)--(-1,3);
\node at (-5,0)[label=west:$A$] {};
\node at (-1,2.5)[label=east:$B$] {};
\node at (-1,-2.5)[label=east:$B$] {};
\node at (-3,2.5)[label=west:{\color{myblue}$D_{1/2}(A,B)$}] {};
\node at (0.2,0)[label=west:{\color{myred}$M_{1/2}(A,B)$}] {};
\draw[thick](1,0)--(2,0);

\filldraw[color=mygreen,fill=mygreen!20, thick](3.5,1)--(2.5,1)--(2.5,2)--(3.5,2)--cycle;
\filldraw[color=mygreen,fill=mygreen!20, thick](3.5,-1)--(2.5,-1)--(2.5,-2)--(3.5,-2)--cycle;
\draw[dashed, color=mygreen] (2,0)--(5,2);
\draw[dashed, color=mygreen] (1,1)--(4,3);
\draw[dashed, color=mygreen] (2,0)--(5,-2);
\draw[dashed, color=mygreen] (1,-1)--(4,-3);

\node at (4.5,3)[label=south:$B_1$] {};
\node at (4.5,-3)[label=north:$B_2$] {};
\node at (0.9,0.5)[label=east:$A_1$] {};
\node at (0.9,-0.5)[label=east:$A_2$] {};

\end{tikzpicture}
\end{center}
\caption{On the left-hand side, the sets $D_{1/2}(A,B)$ and $M_{1/2}(A,B)$ differ. On the right-hand side, after partitioning $A$ and $B$, the two coincide.}
\label{fig:graffi_di_leone}
\end{figure}
\noindent 
In this case, one has
\begin{equation}
    \mathscr{L}^2\big(M_{1/2}(A,B)\big)>\mathscr{L}^2\big(D_{1/2}(A,B)\big).
\end{equation}
However, partitioning $A$ and $B$ as in the right-hand side of Figure \ref{fig:graffi_di_leone}, the problem is remedied, indeed  
\begin{equation}
\label{eq:BASKETS}
    M_t(A_1,B_1)=D_t(A_1,B_1) \quad \text{and} \quad   M_t(A_2,B_2)=D_t(A_2,B_2), 
\end{equation}
and now the strategy of the proof of Theorem \ref{thm:cdwithstep} can be employed without changes. The delicate issue is that, in general metric measure spaces, we are not able to find a suitable partition for arbitrary Borel sets $A$ and $B$, in such a way \eqref{eq:BASKETS} is verified (up to $\m$-null sets) for \emph{each element} of the partition.



\bibliographystyle{alphaabbr}

\end{document}